\documentclass[preprint,12pt]{elsarticle}
\usepackage{amssymb}
\usepackage{latexsym}
\usepackage{amsmath,amsfonts,amssymb}
\usepackage{graphicx}
\newtheorem{theorem}{Theorem}[section]
\newtheorem{lemma}[theorem]{Lemma}
\newtheorem{proposition}[theorem]{Proposition}

\newtheorem{remark}[theorem]{Remark}

\newtheorem{example}[theorem]{Example}
\newtheorem{definition}[theorem]{Definition}

\newtheorem{rem}{Remark}

\newcommand{\thmref}[1]{Theorem~\ref{#1}}

\newcommand{\lemref}[1]{Lemma~\ref{#1}}

\newcommand{\NN}{{\mathbb{N}}}

\newenvironment{proof}{\medskip                    
\noindent{\scshape Proof:}}{\quad $\square$
\medskip}

\def\attr{\mathop{\rm attr}}
\def\kd{\hfil$\square$\linebreak}

\newcommand{\undx}{{\underline x}}
\newcommand{\ovex}{{\overline x}}

\newcommand{\BB}{{\mathbb{B}}}
\def\mbf#1{\mbox{\boldmath$#1$}}
\begin{document}

\title{Characterizing matrices with $\mbf{X}$-simple image eigenspace in max-min semiring}

\author[rvt]{J\'an Plavka\corref{cor}\fnref{fn1}}
\ead{jan.plavka@tuke.sk}

\author[rvt2]{Serge{\u\i} Sergeev\fnref{fn2}}
\ead{sergiej@gmail.com}

\address[rvt]{Department of Mathematics and Theoretical Informatics,  Technical  University,\\
B. N\v emcovej 32, 04200 Ko\v sice, Slovakia}

\address[rvt2]{University of Birmingham, School of Mathematics, Edgbaston B15 2TT, UK}

\cortext[cor]{Corresponding author. Email: jan.plavka@tuke.sk}
\fntext[fn1]{Supported by KEGA grant 032TUKE-4/2013, APVV grant 04-04-12.}

\fntext[fn2]{Supported by EPSRC grant EP/J00829X/1, RFBR grant
12-01-00886.}

\begin{abstract}

A matrix $A$ is said to have $\mbf{X}$-simple image eigenspace if
any eigenvector $x$ belonging to the interval $\mbf{X}=\{x\colon
\undx\leq x\leq\ovex\}$ is the unique solution of the system
$A\otimes y=x$ in $\mbf{X}$. The main result of this paper is a
combinatorial characterization of such matrices in the linear
algebra over max-min (fuzzy) semiring.

The characterized property is related to and motivated by the
general development of tropical linear algebra and interval
analysis, as well as the notions of simple image set and weak
robustness (or weak stability) that have been studied in max-min and
max-plus algebras.
\end{abstract}


\maketitle

\section{Introduction}

This paper is concerned with a problem of max-min linear algebra, which is one of
the sub-areas of tropical mathematics. In a wider algebraic context, tropical mathematics
(also known as idempotent mathematics) can be viewed as mathematical
theory developed over additively idempotent ($a\oplus a$) semirings.
Note that the operation of taking maximum of two numbers is the simplest and
the most useful example of an (additively) idempotent semiring.

Idempotent semirings
can be used in a  range of practical problems related to scheduling
and optimization, and offer many new problem statements to pure
mathematicians. There are
several monographs~\cite{Gol,gm1, HOW, KM:97} and collections
of papers~\cite{LM:05,LS:09} on tropical mathematics and its applications.
Let us also mention some connections between
idempotent algebra and fuzzy sets theory~\cite{DNG},~\cite{DNR}.

In the max-min algebra, sometimes also called the ``fuzzy algebra''~\cite{g1},
the arithmetical operations $a\oplus b:=\max(a,b)$ and $a\otimes b:=\min(a,b)$
are defined over a linearly ordered set. In the present paper, this linearly
ordered set is just the interval $[0,1]=\{\alpha\colon 0\leq\alpha\leq 1\}$.
As usual, the two arithmetical operations are naturally extended to matrices and
vectors.

The development of linear algebra over idempotent semirings was
historically motivated by multi-machine interaction
processes. In these processes we have $n$ machines which work in
stages, and in the algebraic model of their interactive work,
entry $x^{(k)}_i$ of a vector $x^{(k)}\in\BB^n$
where $i\in\{1,\ldots,n\}$ and $\BB$ is an idempotent semiring,
represents the state of machine $i$
after some stage $k$, and the entry $a_{ij}$ of a matrix $A\in\BB(n,n)$, where
$i,j\in\{1,\ldots,n\}$, encodes the influence of the work of machine $j$
in the previous stage on the work of machine $i$ in the current stage. For
simplicity, the process is assumed to be homogeneous, like in
the discrete time Markov chains, so that $A$ does not change from
stage to stage.
Summing up all the influence effects multiplied by the results of previous stages,
we have  $x_i^{(k+1)}=\bigoplus_j a_{ij}\otimes x_j^{(k)}$. In the case of
 $\oplus=\max$
this ``summation'' is often interpreted as waiting till
all the processes are finished and all the
necessary influence constraints are satisfied.

Thus the orbit $x,\, A\otimes x,\ldots A^k\otimes x$, where
$A^k=A\otimes\ldots\otimes A$, represents the evolution of such a process.
Regarding the orbits, one wishes to know the set of starting vectors from which
a given objective can be achieved. One of the most natural objectives in tropical
algebra, where the ultimate periodicity of the orbits often occurs, is to
arrive at an eigenvector. The set of starting vectors from
which one reaches an eigenvector of $A$ after a finite number of stages, is
called attraction set of $A$~\cite{But:10}, \cite{Ser-11}.
In general, attraction set contains the set of all eigenvectors, but it can be also
as big as the whole space. This leads us, in turn, to another question:
in which case is attraction set precisely the same as the set of all
eigenvectors?
Matrices with this property are called weakly robust or weakly stable~\cite{bss}.

In the special case of max-min algebra which
we are going to consider, it can be argued that an orbit can stabilize at a fixed
point ($A\otimes x=x$), but not at an
eigenvector with an eigenvalue different from unity.
Therefore, {\em by eigenvectors of $A$ we shall mean the fixed points of $A$
(satisfying $A\otimes x=x$)}.

In terms of the systems $A\otimes x=b$,
weak robustness (with eigenvectors understood
as fixed points) is equivalent to the following condition: every eigenvector
$y$ belongs to the {\em simple image set} of $A$, that is,  for every eigenvector $y$,
the system $A\otimes x=y$ has unique solution $x=y$.

In the present paper, we consider an interval version of this
condition. Namely, we describe matrices $A$ such that for any
eigenvector $y$ belonging to an interval
$\mbf{X}=[\undx,\ovex]:=\{x\in\BB^n;\, \undx\leq x\leq \ovex\}$ the
system $A\otimes x=y$ has a unique solution $x=y$ in $\mbf{X}$. This
is what we mean by saying that ``$A$ has $\mbf{X}$-simple image
eigenspace''. In Theorem~\ref{Th_WXRmaxminplus2}, which is the main
result of the paper, we show that under a certain natural condition,
$A$ has $\mbf{X}$-simple image eigenspace if and only if it
satisfies a nontrivial combinatorial criterion, which makes use of
threshold digraphs and to which we refer as ``$\mbf{X}$-conformism''
(see Definition~\ref{def:Xconf}).

The next section will be occupied by some definitions and notation
of the max-min algebra, leading to the discussion of weak $\mbf{X}$-
robustness and $\mbf{X}$-simple image eigenvectors.
Section~\ref{s:Xsimple} is devoted to the main result
(characterizing matrices with $\mbf{X}$-simple image eigenspace),
and its rather technical combinatorics. In Section~\ref{s:upward} we
prove a particular property of $\mbf{X}$-simple image eigenvectors,
to which we refer as ``upwardness''. This property states that if
$\alpha\otimes x$ is an $\mbf{X}$-simple image eigenvector, then so
is $\beta\otimes x$ for each $\beta\geq \alpha$.

Let us conclude with a brief overview of the
works on max-min algebra to which this
paper is related.
The concepts of robustness in max-min algebra were introduced and
studied in  \cite{PS}.  Following that work, some equivalent conditions
and efficient algorithms were
presented in \cite{mmp}, \cite{ms}, \cite{P1}.
In particular, see \cite{P1} for
some polynomial procedures checking  the weak robustness (weak stability)
in max-min algebra.

\section{Preliminaries}

\subsection{Max-min algebra and associated digraphs}

Let us denote  the set of all natural numbers by $\NN$.
Let $(\BB,\leq)$ be a bounded linearly ordered set  with the least element
in $\BB$   denoted by $O$ and  the greatest one by $I$.

A max-min semiring is a set $\BB$ equipped with two binary operations $\oplus=\max$ and $\otimes=\min$, called addition and multiplication,  such that
    $(\BB, \oplus)$ is a commutative monoid with identity element $O$,
            $(\BB, \otimes)$ is a monoid with identity element $I$,
             multiplication left and right distributes over addition and
            multiplication by $O$ annihilates $\BB$.


We will use use the notations $N$ and $ M$
 for the sets of  natural numbers not exceeding $n$ and $m$,
respectively, i.e.,
 $N = \{ 1,\, 2,\, \dots,\,n \}$ and $M = \{ 1,\, 2,\, \dots,\,m \}$.
 The set of $n\times m$ matrices over $\BB$ is denoted by $\BB(n,m)$,
and the set of $n\times 1$ vectors over $\BB$ is denoted by
$\BB(n)$.  If each entry of a matrix
$A\in  \BB(n,n)$ (a vector $x\in\BB(n)$) is equal to $O$ we shall
denote it as $A =O$ ($x=O$).

Let $x=(x_1,\dots,x_n)\in \BB(n)$ and $y=(y_1,\dots,y_n) \in \BB(n)$
be vectors. We write    $x\leq y\ (x<y)$   if $x_i\leq y_i\
(x_i<y_i)$ holds for each $i\in N$.\\

For a matrix $A\in \BB(n,n)$ the symbol $G(A)=(N,E)$ stands for a complete,
arc-weighted digraph associated with $A$, i.e., the node set of $G(A)$
is $N$, and the weight (capacity) of any arc $(i,j)$ is $a_{ij}\geq O$.
For given $h \in \BB$, the \textit{threshold digraph} $G(A,h)$
is the digraph with the node set $N$ and with the arc set
$E = \{(i, j);\ i, j \in N,\ a_{ij}\geq h\}$.
A path in the digraph $G(A)=(N,E)$ is a sequence of nodes
$p=(i_1,\,\ldots,\,i_{k+1})$ such that $(i_j,i_{j+1})\in E$ for
$j=1,\,\ldots,\,k$. The number $k$ is the length of the path $p$ and
is denoted by $l(p)$. If $i_1 = i_{k+1}$, then $p$ is called a cycle and it is
called an elementary cycle if moreover $i_j\neq i_m$
for $j,m=1,\dots,k$.

\subsection{Orbits, eigenvectors and weak robustness}
\label{s:generalities}

 For  $A\in \BB(n,n)$ and $x\in\BB(n)$, the
orbit $O(A,x)$ of $x=x^{(0)}$ generated by $A$ is the sequence
$$x^{(0)},x^{(1)},x^{(2)},\dots,x^{(n)},\dots,$$ where $x^{(r)}=A^r
\otimes x^{(0)}$ for each $r\in\mathbb{N}$.

The operations $\max,\min$  are idempotent, so no new numbers
are created in the process of
generating of an orbit. Therefore any orbit
contains only a finite number of different
vectors. It follows that any orbit starts repeating itself after some time,
in other words, it is ultimately periodic.
The same holds for the power sequence $(A^k; k \in\NN)$.

We are interested in the case when the ultimate period is $1$, or in other
words, when the orbit is ultimately stable. Note that in this case the
ultimate vector of the orbit necessarily satisfies $A\otimes x=x$.
This is the main reason why in this paper by eigenvectors we mean
fixed points. (Also observe that if $x$ is not a fixed point but a more general eigenvector satisfying
$A\otimes x=\lambda\otimes x$, then
$A\otimes x$ is already a fixed point due to the idempotency of multiplication.)

Formally we can define the attraction set $\attr(A)$   as follows
$$\attr(A) = \{x \in \BB(n);\ O(A, x) \cap V
(A)\neq \emptyset\}.$$
The present paper is closely related to the following kind of matrices.

\begin{definition}
Let  $A\in \BB(n,n)$ be a matrix. Then  $A$ is called
weakly robust (or weakly stable), if
$\attr(A)=V(A)$.
\end{definition}

Observe that in general $V(A)\subseteq\attr(A)\subseteq \BB^n$. The matrices for
which $\attr(A)=\BB^n$ are called (strongly) robust or (strongly) stable, as opposed
to weakly robust (weakly stable). The following fact, which holds in max-min algebra
and max-plus algebra alike, is one of the main motivations for our paper.

\begin{theorem}{\rm\cite{PS},\cite{bss}} \label{WS}
Let  $A\in \BB(n,n)$ be a matrix. Then $A$ is weakly robust if and only if
$(\forall x\in \BB(n))[A \otimes x \in V (A) \Rightarrow x\in V (A)].$
\end{theorem}



Let us conclude this section with recalling some information on 1) the greatest
eigenvector and 2) constant eigenvectors in max-min algebra.

Let  $A=(a_{ij})\in \BB(n,n)$ be a matrix in define the greatest eigenvector $x^{\oplus}(A)$
corresponding to a matrix $A$   as
$$x^{\oplus}(A)=\bigoplus\limits_{x\in
V(A)}\hspace{-0.1cm}x.$$

It has been
proved in \cite{y} for a more general structure (distributive lattice) that the greatest eigenvector $x^{\oplus}(A)$ of $A$ exists for each  matrix  $A\in \BB(n,n)$.
The greatest
 eigenvector $x^\oplus(A)$ can be computed by the following
iterative $O(n^2\log n)$ procedure (\cite{c3}).  Let us denote
$x_i^1(A)=\bigoplus\limits_{j\in N} a_{ij}$ for each $i\in N$ and $
x^{k+1}(A)=A\otimes x^{k}(A)\ \ \text{for all}\ \
k\in\{1,2,\dots\}$. Then $x^{k+1}(A)\leq  x^{k}(A)$ and
$x^\oplus(A)=x^n(A)$. Observe that $x_i^{\oplus}(A)\leq\bigoplus\limits_{j\in
N} a_{ij}$ for all $i$.

Next, denote
$$m_A=\bigoplus\limits_{i,j\in N} a_{ij},\ \ \ c(A)=\bigotimes\limits_{i\in N}\
\bigoplus\limits_{j\in N} a_{ij},\ \ \ c^*(A)=(c(A),\dots,c(A))^T\in
\BB(n).$$ It can be checked that $A\otimes c^*(A)=c^*(A)$, since
every row of $A$ contains an entry that is not smaller than
$c^*(A)$. In fact, this condition is necessary and sufficient for a
constant eigenvector to be an eigenvector of $A$. Therefore any
constant vector that is smaller than $c^*(A)$ is also an
eigenvector, and $c^*(A)$ is the largest constant eigenvector of
$A$. However, as $x^{\oplus}(A)$ is the greatest eigenvector of $A$,
we have $c^*(A)\leq x^{\oplus}(A)$.

\subsection{Weak $\mbf{X}$-robustness and $\mbf{X}$-simplicity}

In this section we consider an interval extension of weak robustness and
its connection to $\mbf{X}$-simplicity, the main notion studied in this paper.
We remind that throughout the paper,
$$
\mbf{X}=[\undx,\ovex]=\{x\in\BB^n\colon \undx\leq x\leq\ovex,\}, \quad\text{where
$\undx,\ovex\in\BB^n$}.
$$

Consider the following interval extension of weak $\mbf{X}$-robustness.

\begin{definition}
$A\in\BB^n$ is called
weakly $\mbf{X}$-robust  if $\attr(A)\cap \mbf{X}\subseteq V(A)$.
\end{definition}

The notion of $\mbf{X}$-simplicity, which we will introduce next, is
related to the concept of simple image set~\cite{ButSIS}: by
definition, this is the set of vectors $b$ such that the system
$A\otimes x=b$ has a unique solution, which is usually denoted by
$|S(A,b)|=1$ ($S(A,b)$ standing for the solution set of $A\otimes x
=b$).
 If the only solution of the system $A\otimes x=b$ is $x=b$, then $b$ is
called a {\em simple
image eigenvector.}

If $\mbf{X}=\BB$ then the notion of weak robustness
can be described in terms of simple image eigenvectors:

\begin{proposition} \label{Th_WXRminplus}
Let $A\in\BB(n,n)$. The following are equivalent:
\begin{enumerate}
\item $A$ is weakly robust;
\item $(\forall x\in   V(A))[|S(A,x)|=1]$;
\item Each $x\in V(A)$ is a simple image eigenvector.
\end{enumerate}
\end{proposition}
{\it Proof.} We will only prove the equivalence between the
first two claims (the other equivalence being evident). Suppose that there is $x\in V(A)$ such that $|S(A,x)|>1$ (notice that $|S(A,x)|\geq 1$ for each $x$ because of $x\in V(A)$). Then there is at least one solution $y$ of the system $A\otimes y=x$ and $y\neq x$. Using \thmref{WS} we get
$A\otimes(A\otimes y)=A\otimes x=x$ and $A\otimes y=x\neq y$, this is a contradiction.\\
The converse implication trivially follows.\kd

This motivates us to consider an interval version of
simple image eigenvectors.

\begin{definition} Let $A = (a_{ij} )\in \BB(n,n)$.
\begin{enumerate}
\item
An eigenvector $x\in V(A)\cap\mbf{X}$ is called an
$\mbf{X}$-simple image eigenvector if $x$ is the unique solution of the
equation $A\otimes y=x$ in interval $\mbf{X}$.
\item
Matrix $A$ is said to have $\mbf{X}$-simple image eigenspace if any
$x\in   V(A)\cap\mbf{X}$ is an $\mbf{X}$-simple image eigenvector.
\end{enumerate}
\end{definition}

\if{
Let us also provide a more direct characterization of $A$ having
$X$-simple image eigenspace.
\begin{lemma}
$A\in\BB(n,n)$ has simple image eigenspace if and only if
$(\forall x\in \mbf{X})[A\otimes x \in
V (A) \Rightarrow x\in V (A)].$
\end{lemma}
\begin{proof}
The ``only if'' part.
Let $A\in\BB(n,n)$ have $\mbf{X}$-simple image eigenspace, and
let $x\in\mbf{X}$ and $A\otimes x\in V(A)$, thus $A\otimes x=y$ where
$y\in V(A)$
Then
\end{proof}
}\fi

\begin{theorem} \label{WXR}
Let   $A\in \BB(n,n)$ be a matrix and $\mbf{X}=[\undx, \ovex] \subseteq \BB(n)$
be an interval vector.
\begin{enumerate}
\item
If $A$ is weakly $\mbf{X}$-robust then
$A$ has $\mbf{X}$-simple image eigenspace.

\item
If $A$ has $\mbf{X}$-simple image eigenspace and if
$\mbf{X}$ is invariant under $A$ then $A$ is
weakly $\mbf{X}$-robust.

\end{enumerate}
\end{theorem}
{\it Proof.} (i) Suppose  that $A$ is weakly $\mbf{X}$-robust and
$x\in V(A)\cap\mbf{X}$. If the system $A\otimes y=x$ has a solutions $y\neq x$ in
$\mbf{X}$, then $y$ is not an eigenvector but belongs to $\attr(A)\cap\mbf{X}$,
which contradicts the weak $\mbf{X}$-robustness.

(ii) Assume that $A$ has $X$-simple eigenspace and
$x$ is an arbitrary element of $\attr(A)\cap\mbf{X}$.
As $\mbf{X}$ is invariant under $A$, we have that $A^k\otimes x\in\mbf{X}$
for all $k$. Then $A^k\otimes x\in
V(A)$ for some $k$ implies
$A^{k-1}\otimes x=A^k\otimes x\in V(A)$,...,
$x\in V(A)$.
\kd

As $A$ is order-preserving, the invariance of $\mbf{X}$ under $A$ admits the following
simple characterization:

\begin{proposition}
\label{p:inv}
$\mbf{X}$ is invariant under $A$ if and only if $A\otimes \undx\geq \undx$ and
$A\otimes\ovex\leq\ovex$.
\end{proposition}

Thus the $\mbf{X}$-simplicity is a necessary condition for
weak $\mbf{X}$-robustness. It is also sufficient if
the interval $\mbf{X}$ is invariant under $A$, i.e.,
$\undx\leq A\otimes\undx$ and  $A\otimes\ovex\leq\ovex$.

\section{ $\mbf{X}$-simple image eigenspace and $\mbf{X}$-conformism}

\label{s:Xsimple}

The purpose of this section is to define the condition for matrix $A$ which will
ensure that each eigenvector $x\in V(A)\cap\mbf{X}$ is an $\mbf{X}$-simple image
eigenvector.

\begin{definition}
\label{def:levelperm} A matrix $A= (a_{ij} )\in\BB(n,n)$ is called a generalized level
$\alpha$-permutation matrix (abrr. level
$\alpha$-permutation)  if all  entries greater than or equal to $\alpha$ of $A$ lie on
disjoint elementary cycles covering all the nodes. In other words,
the threshold digraph $G(A,\alpha)$
 is the set of disjoint elementary cycle
containing all nodes.
\end{definition}

Let us also define the following quantity:

\begin{equation}
\label{e:gax}
\begin{split}
\gamma(A,\ovex)&=\min(c(A),\min_{i\in N} \ovex_i),\quad\text{where
$c(A)=\min_{i\in N}\max_{j\in N} a_{ij}$}\\
\gamma^*(A,\ovex)&=(\gamma(A,\ovex),\ldots,\gamma(A,\ovex)).
\end{split}
\end{equation}

Since $\gamma^*(A,\ovex)$ is a constant vector such that each row of
$A$ contains an entry not smaller than $\gamma(A,\ovex)$, we obtain
$A\otimes \gamma^*(A,\ovex) =\gamma^*(A,\ovex)$
(i.e.,$\gamma^*(A,\ovex)\in V(A)$).

\begin{lemma}\label{l:cAperm} Let $A = (a_{ij} )\in \BB(n,n)$ be a  matrix,
$\mbf{X}=[\undx, \ovex] \in \BB(n)$ be an interval vector. Assume
that $\undx<c^*(A)$ and $\max_{i\in N} \undx_i<\min_{i\in N}
\ovex_i$. Then, if $A$ has $\mbf{X}$-simple image eigenspace then $A$
is level $\gamma(A,\ovex))$-permutation.
\end{lemma}
\begin{proof} 
For a contrary suppose that, under the given conditions, $A$ is not
level $\gamma(A,\ovex))$-permutation. We shall look for two solutions
of $A\otimes y=\gamma^*(A,\ovex)$. One solution is $\gamma^*(A,\ovex)\in V(A)\cap\mbf{X}$.
Since $A$ is not level $\gamma(A,\ovex)$-permutation and each
row of $A$ contains at least one element $a_{ij}\geq \gamma(A,\ovex)$ we shall consider two cases.\\
Case 1. $(\exists k\in N)[\max\limits_{s\in N}
a_{sk}<\gamma(A,\ovex)]$. The second solution $y'\in \mbf{X}$ is
$$y'_i=\begin{cases}
\undx_i, & \text{  if  }i= k\\
\gamma(A,\ovex), & \text{ otherwise},
\end{cases}
$$
since we have $a_{sk}<\gamma(A,\ovex)$ for all $s$, implying that
the terms $a_{sk}\otimes y'_k$ are unimportant and $y'_k$ can be set
to any admissible value.

Case 2. $(\forall k\in N)[\max\limits_{s\in N}a_{sk}\geq
\gamma(A,\ovex)]$ and $(\exists i,j,k\in N)[a_{ij}\geq
\gamma(A,\ovex)$ and $a_{ik}\geq \gamma(A,\ovex)].$ Then there is
$v\in N$ such that $(\forall i\in N)[\max\limits_{j\in
N\setminus\{v\}}a_{ij}\geq \gamma(A,\ovex)]$ and the second solution
$y'\in \mbf{X}$ can be defined as follows
$$y'_i=\begin{cases}
\undx_i, & \text{  if  }i= v\\
\gamma(A,\ovex), & \text{ otherwise},\\
\end{cases}
$$
since attainment of the maximum value in every row of $A\otimes y$
by other terms than $a_{sv}\otimes y_v$ makes these terms redundant,
so that $y_v$ can be replaced by any admissible value $y'_v<y_v$.

In both cases we obtained a contradiction with $A$ having $\mbf{X}$-simple
image eigenspace.
\end{proof}

\begin{definition}
\label{def:Xconf} Let $\mbf{X}=[\undx, \ovex] \subseteq \BB(n)$ be
an interval vector   such that $\undx<c^*(A)$ and $\max_{i\in
N}\undx_i<\min_{i\in N} \ovex_i$ and $A = (a_{ij} )\in \BB(n,n)$ be
a level $\gamma(A,\ovex)$-permutation matrix. Let $(i_1,\dots,i_n)$
be a permutation of $N$ such that $a_{i_j i_{j+1}}\geq
\gamma(A,\ovex)$ and $(i_1,\dots,i_n)=
(i_1^1,\dots,i_{s_1}^1)\dots(i_1^k,\dots,i_{s_k}^k)$
$(c_u=(i_1^u,\dots,i_{s_u}^u) $
is an elementary cycle in digraph $G(A,\gamma(A,\ovex))$, $u=1,\dots,k)$. Then\\

\hspace{-.81cm} vectors $e_{\undx}=(e_1,\dots,e_n)^T$ and $f_{\ovex}=(f_1,\dots,f_n)^T$
 are called $\undx$-vector of $A$ and $\ovex$-vector of $A$ if
$$e_i=\max\limits_{v\in c_u} \undx_v
\text{  and  }f_i=\min\limits_{v\in c_u} \ovex_v\otimes x^\oplus_v(A),$$
respectively,   for  $i\in c_u$, $u\in\{1,\dots,k\}$\\

\hspace{-.61cm}and \\

 \hspace{-.61cm}matrix $A $  is called   $\mbf{X}$-conforming  if
\begin{enumerate}
\item
$\undx_{i_{j+1}}< e_{i_{j+1}} \Rightarrow a_{i_jk}< e_{i_j} \text{  for  }  k\neq i_{j+1},\ k\in N$
\item
$\undx_{i_{j+1}}= e_{i_{j+1}} \Rightarrow a_{i_jk}\leq e_{i_j} \text{  for  }  k\neq i_{j+1},\ k\in N$
\item
$a_{i_ji_{j+1}}=\min\limits_{(k,l)\in c_u} a_{kl}=x^\oplus_{i_{j+1}}(A)=f_{i_{j+1}}\Rightarrow \ovex_{i_{j+1}}\leq x^\oplus_{i_{j+1}}(A)$.
\end{enumerate}
\end{definition}

\begin{remark}
\label{r:cycleq} {\rm Notice that $e_{i_{j}}=e_{i_{j+1}}$ and
$f_{i_{j}}=f_{i_{j+1}}$ by definition of $e_\undx$ and $f_\ovex$
(nodes $i_{j},i_{j+1}$ are lying in the same cycle $c_u$). Notation
$(k,l)\in c_u$ means that the edge $(k,l)$ is lying in $c_u$.}
\end{remark}

\begin{remark}
\label{r:f} {\rm Observe that $f_i\leq\bigoplus_{j\in N} a_{ij}$ for
all $i$, since the same holds for $x^{\oplus}(A)$ (the end of
Section~\ref{s:generalities}.}
\end{remark}

\begin{example}
Let us consider  $\BB=[0,10]$, $\lambda=10$
and
\begin{displaymath}A=\left(\begin{array}{cccc}
4&4&4&\bf{5}\\
2&2&\bf{7}&2\\
3&\bf{8}&3&3\\
\bf{7}&3&3&3
\end{array}\right),\ \
\undx=
\left(\begin{array}{c}
2\\
3\\
2\\
4
\end{array}\right),\ \
\ovex=
\left(\begin{array}{c}
7\\
9\\
6\\
5
\end{array}\right).
\end{displaymath}
Matrix $A$ is level $5$-permutation with $c_1=(i_1,i_2)=(1,4)$, $c_2=(i_3,i_4)=(2,3)$ and
$x^\oplus(A)=(5,7,7,5)^T$. Vectors  $e_{\undx}$ and $f_\ovex$ have the following coordinates
$$e_1=e_4=\max(\undx_1,\undx_4)=4,\;
e_2=e_3=\max(\undx_2,\undx_3)=3,$$
$$f_1=f_4=\min(\ovex_1\otimes x^\oplus_1,\ovex_4\otimes x^\oplus_4)=5,\;
f_2=f_3=\min(\ovex_2\otimes x^\oplus_2,\ovex_3\otimes x^\oplus_3)=6,$$
thus $e_\undx=(4,3,3,4)^T$ and $f_\ovex=(5,6,6,5)^T$.

Now,  we shall argue that $A$ is   $\mbf{X}$-conforming,

$$i_1=1,i_2=4;\ \  \undx_1<e_1\Rightarrow a_{4j}< e_4\ (\forall j\neq 1),$$
$$i_2=4, i_1=1;\ \ \undx_4=e_4 \Rightarrow a_{1j}\leq e_1\ (\forall j\neq 4),$$
$$i_3=2,i_4=3;\ \  \undx_2=e_2\Rightarrow a_{3j}\leq e_3\ (\forall j\neq 2),$$
$$i_4=3,i_3=2;\ \  \undx_3<e_3 \Rightarrow a_{2j}< e_2\  (\forall j\neq 3)$$
and
$$a_{14}=5=\min\limits_{(k,l)\in c_1=(14)} a_{kl}=x^\oplus_{4}(A)\Rightarrow \ovex_4=5\leq x^\oplus_4(A)=5.$$
Hence matrix $A$ is    $\mbf{X}$-conforming.
\end{example}

\begin{lemma}\label{l:eigspace} Let $A = (a_{ij} )\in \BB(n,n)$ be a  matrix, and let $\mbf{X}=[\undx, \ovex]
\in \BB(n)$ be an interval vector. Assume that $\undx<c^*(A)$,
$\max_{i\in N}\undx_i<\min_{i\in N} \ovex_i$ and that $A$ is
$\mbf{X}$-conforming. Then
\begin{enumerate}
\item $ \undx\leq e_{\undx}< \gamma^*(A,\ovex)$, $\gamma^*(A,\ovex)\leq f_{\ovex}\leq\ovex$ and   $e_{\undx},f_\ovex\in V(A)\cap\mbf{X}$,
\item $(\forall x\in\mbf{X}\cap V(A))[e_{\undx}\leq x\leq f_\ovex]$,
\item $V(A)\cap\mbf{X}=\{(x_1,\dots,x_n)^T;\
x_i=\alpha_u\in [e_i, f_i]\text{ for }i\in c_u,\,1\leq u\leq k\}.$

\end{enumerate}
\end{lemma}
\begin{proof} Let us first observe that by Lemma~\ref{l:cAperm} $A$
is a level $\gamma(A,\ovex)$-permutation matrix.

(i) The inequalities $ \undx\leq e_{\undx}< \gamma^*(A,\ovex)$
follow from the conditions $\undx<c^*(A)$, $\max_{\i\in N}
\undx_i<\min_{i\in N}\ovex_i$  and the definition of $e_{\undx}$. To
obtain $\gamma^*(A,\ovex)\leq f_{\ovex}\leq\ovex$, recall that
$x^{\oplus}(A)\geq \gamma^*(A,\ovex)$ as $x^{\oplus}(A)$ is the
largest eigenvector, and that $\overline{x}\geq \gamma^*(A,\ovex)$
by~\eqref{e:gax}, implying $f_{\ovex}\geq \gamma^*(A,\ovex)$.

To show that $e_{\undx},f_{\ovex}\in V(A)$ we need to prove that
$A\otimes e_{\undx}=e_{\undx},\ A\otimes f_{\ovex}=f_{\ovex}.$ As
matrix $A$ is level $\gamma(A,\ovex)$-permutation, for each $i \in
N$ there is $j\in N$ such that $a_{ij}\geq \gamma(A,\ovex)$.

To prove $A\otimes e_{\undx}=e_{\undx}$ observe that
$$(A\otimes e_{\undx})_{i }=\bigoplus\limits_{t\neq j}a_{i t}\otimes e_{t}\oplus a_{ij}\otimes e_j=a_{ij}\otimes e_j=e_j=e_i$$
because $i,j$ lie in the same cycle ($e_i=e_j$) and $a_{ij}\geq
\gamma(A,\ovex)>e_i\geq a_{it}$ for all $t\neq j$ by the definition
of $\mbf{X}$-conforming matrix.

To prove $A\otimes f_{\ovex}=f_{\ovex}$ observe that
$$(A\otimes f_{\ovex})_{i }=\bigoplus\limits_{t\in N}a_{i t}\otimes f_{t}=a_{ij}\otimes f_j=f_j=f_i$$
because of  $i,j$ are lying in the same cycle ($f_i=f_j$),
$a_{ij}=\bigoplus\limits_{t\in N}a_{it}\geq x^{\oplus}_i(A)\geq \ovex_i\otimes x^\oplus_i(A)\geq f_i$
and $a_{it}<\gamma(A,\ovex)\leq f_i=f_j$ for $t\neq j$.\\

(ii) Suppose that $(\exists x\in\mbf{X}\cap V(A)[e_{\undx}\nleq x]$,
i.e., there is at least one index $i\in N$ such that $\undx_i\leq
x_i<e_i$. Since $A$ is level $\gamma(A,\ovex)$-permutation and
$i\in N$, then there is a cycle $c=\{i_1,\dots,i_{s}\}$  such that
$i_1=i\in c$,
$$a_{i_r i_{r+1}}\geq \gamma(A,\ovex)>e_i>x_i\text{  and  }e_{i_r}=
\max_{k=i_1,\dots,i_s}\undx_k\text{  for  }r=1,\dots,s.$$

If $s=1$ we immediately obtain a contradiction with the definition of the vector $e_{\undx}$. Suppose
now that $s\geq 2$, then for the eigenvector $x=(x_1,\dots,x_n)$ we have
$$ x_{i_1}=(A\otimes  x )_{i_1}=\bigoplus\limits_{t\in N}a_{i_1t}\otimes x_{t}\geq a_{i_1 i _{2}}\otimes x_{i _{2}}=x_{i _{2}}$$
because of $a_{i_1 i_{2}}\geq \gamma(A,\ovex)>x_{i_1}$ and we obtain
$x_{i_1}\geq x_{i_2}$,
$$ x_{i_2}=(A\otimes  x )_{i_2}=\bigoplus\limits_{t\in N}a_{i_2t}\otimes x_{t}\geq a_{i_2 i _{3}}\otimes x_{i _{3}}=x_{i _{3}}$$
because of $a_{i_2 i_{3}}\geq \gamma(A,\ovex)>x_{i_1}\geq x_{i_2}$
and we obtain $x_{i_2}\geq x_{i_3}$. Proceeding in the same way for
$x_{i_3},\dots,x_{i_s}$ we have
$x_{i_1}=\dots=x_{i_s}<e_{i_1}=\ldots=e_{i_s}.$ However, this
implies
$x_{i_k}<\undx_{i_k}$ for some $i_k$, which is a contradiction. \\

Suppose that $(\exists x\in\mbf{X}\cap V(A)[x\nleq f_{\ovex}]$,
i.e., there is at least one index $i\in N$ such that $f_i< x_i\leq
\ovex_i$. Since $x$ is an eigenvector of $A$ the equality
$\bigoplus\limits_{t\in N}a_{vt}\geq x_v$ holds for each $v\in N$.
Moreover using part (ii) and that $A$ is $\mbf{X}$-conforming, we
have $a_{it}\leq e_i\leq f_i<x_i$ for each
$(i,t)\notin\{(i_1,i_2),\dots,(i_s,i_1)\}$. Let $i=i_1\in
c=\{i_1,\dots,i_{s}\}$. Then
$$x_{i_1}=(A\otimes  x )_{i_1}=\bigoplus\limits_{t\in N}a_{i_1t}\otimes x_{t}
= a_{i_1 i _{2}}\otimes x_{i _{2}}\Rightarrow x_{i _{2}}\geq x_{i_1},$$
$$x_{i_2}=(A\otimes  x )_{i_2}=\bigoplus\limits_{t\in N}a_{i_2t}\otimes x_{t}
= a_{i_2 i _{3}}\otimes x_{i _{3}}\Rightarrow x_{i _{3}}\geq x_{i_2},$$
since $a_{i_1t}< x_{i_1}$ for $t\neq i_2$ and $a_{i_2t}\leq e_{i_2}\leq f_{i_2}=f_{i_1}<x_{i_1}$ for $t\neq i_3$.
Proceeding in the same way for $x_{i_3},\dots,x_{i_s}$ we obtain that
\begin{equation*}
x_{i_1}=\dots=x_{i_s}>f_{i_1}=\ldots=f_{i_s}
=\min\limits_{k=i_1,\dots,i_s}(\ovex_k,\, x^\oplus_k(A))\text{  for  }r=1,\dots,s.
\end{equation*}
This implies that $x_{i_l}>\ovex_{i_l}$ or $x_{i_l}>x^{\oplus}_{i_l}(A)$ for some $i_l$:
in both cases, a contradiction.

(iii) By part (ii) each $x\in V(A)\cap\mbf{X}$ satisfies $e_{\undx}\leq x\leq f_{\ovex}$, and it remains to show that
$$x=(x_1,\dots,x_n)^T;\
x_i=\alpha_u\ \text{ for }i\in c_u\text{ and }1\leq u\leq k.$$ As
$A$ is $\mbf{X}$-conforming, $A$ is also a level
$\gamma(A,\ovex)$-permutation matrix such that $(i_1,\dots,i_n)$ is
a permutation of $N$ with $a_{i_ji_{j+1}}\geq \gamma(A,\ovex)$,
$$(i_1,\dots,i_n)=(i_1^1,\dots,i_{s_1}^1)\dots(i_1^k,\dots,i_{s_k}^k),$$
$c_u=(i_1^u,\dots,i_{s_u}^u), u=1,\dots,k $ are the elementary
cycles in $G(A,\gamma(A,\ovex))$ and $a_{rs}\leq
e_r<\gamma(A,\ovex)$ for $(r,s)\notin\{(i_1,i_2),\dots
(i_{n-1},i_n),(i_n,i_1)\}$.

Suppose that   $x\in V(A)\cap\mbf{X}$. Then  $e_\undx\leq x\leq
f_\ovex$ by (ii), and without loss of generality let us assume
$u=1$, that is, $c_1=(i_1^1,\ldots,i_{s_1}^1)$.

We shall consider two cases.\\
Case 1. $x_{i_1^1}=e_{i_1^1}$. Then we have that
$$x_{i_1^1}(=e_{i_1^1})=(A\otimes x)_{i_1^1}=\bigoplus\limits_{t\in N}a_{i_1^1t}\otimes x_{t}= \bigoplus\limits_{t\neq i_2^1}a_{i_1^1t}\otimes x_{t}\oplus a_{i_1^1i_2^1}\otimes x_{i_2^1} \geq a_{i_1^1i_2^1}\otimes x_{i_2^1}.$$
Hence we have that $x_{i_1^1}\geq  x_{i_2^1}$ because of $
a_{i_1^1i_2^1}\geq \gamma(A,\ovex)>e_{i_1^1}=x_{i_1^1}.$ For the
index $i_2^1$ we get
$$x_{i_2^1} =(A\otimes x)_{i_2^1}=\bigoplus\limits_{t\in N}a_{i_2^1t}\otimes x_{t}= \bigoplus\limits_{t\neq i_3^1}a_{i_2^1t}\otimes x_{t}\oplus a_{i_2^1i_3^1}\otimes x_{i_3^1}\geq a_{i_2^1i_3^1}\otimes x_{i_3^1}. $$
Hence we have that $x_{i_2^1}\geq  x_{i_3^1}$ because of $
a_{i_2^1i_3^1}\geq \gamma(A,\ovex)>e_{i_1^1}=x_{i_1^1}\geq
x_{i_2^1}.$ Proceeding in the same way we get
$x_{i_1^1}=\dots=x_{i^1_{s_1}}=e_{i_1^1}\in [e_{i_1^1},f_{i_1^1}]$.\\

Case 2. $x_{i_1^1}>e_{i_1^1}.$ Then we get
$$x_{i_1^1}=\bigoplus\limits_{t\in N}a_{i_1^1t}\otimes x_{t}= \bigoplus\limits_{t\neq i_2^1}a_{i_1^1t}\otimes x_{t}\oplus a_{i_1^1i_2^1}\otimes x_{i_2^1}=a_{i_1^1i_2^1}\otimes x_{i_2^1}$$
because inequalities $x_{i_1^1}>e_{i_1^1}\geq a_{i_1^1t}$ holds for each $t\neq i_2^1$
since $A$ is $\mbf{X}$-conforming. Thus, we obtain
$x_{i_1^1}=a_{i_1^1i_2^1}\otimes x_{i_2^1} $ and hence $x_{i_1^1}\leq x_{i_2^1} $. Similarly for $x_{i_2^1} $ we get
$$x_{i_2^1}=\bigoplus\limits_{t\in N}a_{i_2^1t}\otimes x_{t}= \bigoplus\limits_{t\neq i_3^1}a_{i_2^1t}\otimes x_{t}\oplus a_{i_2^1i_3^1}\otimes x_{i_3^1} =a_{i_2^1i_3^1}\otimes x_{i_3^1}$$
because of $x_{i_2^1}\geq x_{i_1^1}>e_{i_1^1}\geq a_{i_1^1t}$ for
each $t\neq i_3^1$ and  we obtain $x_{i_2^1}=a_{i_2^1i_3^1}\otimes
x_{i_3^1} $. Hence $x_{i_2^1}\leq x_{i_3^1} $. Proceeding in the
same way we get $x_{i_1^1}=\dots=x_{i^1_{s_1}}=\alpha_1\in
[e_{i_1^1},f_{i_1^1}]$.
\end{proof}

\begin{rem}\label{rem}
{\rm By \lemref{l:eigspace} (iii) it follows that the structure of
each eigenvector $x\in V(A)\cap\mbf{X}$ of a given
$\mbf{X}$-conforming matrix $A$ depends on elementary cycles in
$G(A,\gamma(A,\ovex))$ and all entries of $x$ corresponding to the
same cycle have an equal value.}
\end{rem}

\begin{theorem}\label{Th_WXRmaxminplus2}
Let $A = (a_{ij} )\in \BB(n,n)$ be a  matrix, $\mbf{X}=[\undx,
\ovex] \in \BB(n)$ be an interval vector. Assume that $
\undx<c^*(A)$ and $\max_{i\in N} \undx_i<\min_{i\in N} \ovex_i$.
Then $A$ has $\mbf{X}$-simple image eigenspace if and only if  $A$
is an $\mbf{X}$-conforming matrix.
\end{theorem}
\begin{proof}
The ``only if'' part:  As the matrix $A$ has $\mbf{X}$-simple image
eigenspace, by Lemma~\ref{l:cAperm}, matrix $A$ is level
$\gamma(A,\ovex)$-permutation. Suppose that $(i_1,\dots,i_n)$ is a
permutation of $N$ such that $a_{i_ji_{j+1}}\geq \gamma(A,\ovex) $
and
$$(i_1,\dots,i_n)= (i_1^1\dots
i_{s_1}^1)\dots(i_1^l,\dots,i_{s_l}^l),$$ where
$c_u=(i_1^u,\dots,i_{s_u}^u),u=1,\dots,l $ is an elementary cycle in
$G(A,c(A))$.

 For the sake of a contradiction suppose
that $(\exists u\in\{1,\dots,l\})(\exists i_r^u\in c_u=(i_1^u,\dots,i_{s_u}^u))$ such that

$$\undx_{i_{r+1}^u}< e_{i_{r+1}^u} \text{ and } a_{i_r^uk}\geq e_{i_r^u} \text{  for some }  k\neq i_{r+1}^u,k\in N$$
or
$$\undx_{i_{r+1}^u}= e_{i_{r+1}^u} \text{ and } a_{i_r^uk}> e_{i_r^u} \text{  for some }  k\neq i_{r+1}^u,k\in N$$
or
$$a_{i_ji_{j+1}}=\min\limits_{(k,l)\in c_u} a_{kl}=x^\oplus_{i_{j+1}}(A)=f_{i_{j+1}}\text{ and }\ovex_{i_{j+1}}> x^\oplus_{i_{j+1}}(A).$$

We shall consider three cases and
for each case we shall construct an eigenvector
$e'\in V(A)\cap\mbf{X}$ with $|S(A,e')\cap\mbf{X}|\geq 2$. \\
Denote  $d_u=\max\limits_{a_{tv}<\gamma(A,\ovex);t\in c_u,v\in
N}a_{tv}(=a_{i_p^uv})$. The first two cases will be treated
simultaneously.

\noindent
Case 1.  $\undx_{i_{r+1}^u}< e_{i_{r+1}^u}$ and
$a_{i_r^uk}\geq e_{i_r^u}(\Leftrightarrow
\undx_{i_{r+1}^u}< e_{i_{r+1}^u}=e_{i_r^u}\leq a_{i_r^uk})$.\\
Case 2.  $\undx_{i_{r+1}^u}=e_{i_{r+1}^u}$ and $a_{i_r^uk}> e_{i_r^u}(\Leftrightarrow
\undx_{i_{r+1}^u}=e_{i_{r+1}^u}=e_{i_r^u}< a_{i_r^uk})$.\\

Define vector $e'$ as follows
$$
e'_i=\begin{cases}
d_u, & \text{  if  } i\in c_u\\
\gamma(A,\ovex), & \text{  otherwise}.
\end{cases}
$$
We shall show that $A\otimes e'=e'$. Since the matrix $A$ is level
$\gamma(A,\ovex)$-permutation, the equalities $(A\otimes
e')_i=(A\otimes \gamma^*(A,\ovex))_i=\gamma(A,\ovex)=e'_i$ hold for
$i\notin c_u$, so it suffices to show that
 $$(A\otimes e')_{i_t^u}=e'_{i_t^u} \text{  for  }t=1,\dots,s_u.$$
Using the definition of $d_u$ and $e'$ we obtain
 $$(A\otimes e')_{i_t^u}=\bigoplus\limits_{j\neq i_{t+1}^u}a_{i_t^uj}\otimes e'_j\oplus a_{i_t^ui_{t+1}^u}\otimes e'_{i_{t+1}^u}\leq d_u\oplus a_{i_t^ui_{t+1}^u}\otimes d_u=d_u= e'_{i_t^u}
$$
and
 $$(A\otimes e')_{i_t^u}=\bigoplus\limits_{j\in N}a_{i_t^uj}\otimes e'_j\geq a_{i_t^ui_{t+1}^u}\otimes e'_{i_{t+1}^u}= a_{i_t^ui_{t+1}^u}\otimes d_u=d_u= e'_{i_t^u}.
$$

In particular, we have
$$
a_{i_t^ui_{t+1}^u}\otimes e'_{i_{t+1}^u}=e'_{i_t^u},\quad t=1,\ldots,s_u.
$$

To obtain a contradiction we shall show that the system $A\otimes y=e'$ has at least
two solutions, which will be denoted by $y',y''$. We have $a_{i_r^u k}\geq e_{i_r^u}$ and
$\undx_{i_{r+1}^u}<e_{i_{r+1}^u}$ (case 1), or $a_{i_r^u k}> e_{i_r^u}$ and
$\undx_{i_{r+1}^u}=e_{i_{r+1}^u}$ (case 2), and $a_{i_p^u v}=d_u$ for some indices $p$ and $v$.
Define
$$y'=e',\qquad
y''_i=\begin{cases}
\undx_i, & \text{  if  } i=i_{p+1}^u\\
e'_i, & \text{  otherwise}.
\end{cases}
$$
We need to show that $y'_{i_{p+1}^u}>y''_{i_{p+1}^u}$, to make sure
that $y'$ and $y''$ are actually different in this position, and
$y'\geq y''$. Next we also need to show that $A\otimes y''\geq e'$,
hence $A\otimes y''=e'$.

To see the difference, observe that if $p=r$, $a_{i_r^u k}\geq
e_{i_r^u}$ and $\undx_{i_{r+1}^u}<e_{i_{r+1}^u}$ then
$$
y''_{i_{r+1}^u}=\undx_{i_{r+1}^u}<e_{i_{r+1}^u}=e_{i_r^u}\leq a_{i_r^uk}\leq d_u=e'_{i_{r+1}^u},
$$
if $p=r$, $a_{i_r^u k}> e_{i_r^u}$ and
$\undx_{i_{r+1}^u}=e_{i_{r+1}^u}$ then
$$
y''_{i_{r+1}^u}=\undx_{i_{r+1}^u}=e_{i_{r+1}^u}=e_{i_r^u}< a_{i_r^uk}\leq d_u=e'_{i_{r+1}^u},
$$
and if $p\neq r$ then
$$
y''_{i_{p+1}^u}=\undx_{i_{p+1}^u}\leq e_{i_{p+1}^u}=e_{i_p^u}\leq a_{i_p^uk}< d_u=e'_{i_{p+1}^u}.
$$
By the definition of $d_u$ and $e'$ and since $A$ is a level
$\gamma(A,\ovex)$-permutation matrix,  we have $(A\otimes
y'')_{i_t^u}=a_{i_t^u i_{t+1}^u} e'_{i_{t+1}^u}=e'_{i_t^u}$ for each
$t\neq p$. For $t=p$ we obtain the following inequalities
 $$(A\otimes y'')_{i_p^u}=\bigoplus\limits_{j\neq i_{p+1}^u}a_{i_p^uj}\otimes y''_j\oplus
a_{i_p^ui_{p+1}^u}\otimes y''_{i_{p+1}^u}\geq
 a_{i_p^u v}\otimes e'_v\geq d_u=e'_{i_p^u},$$
 where $k\neq i_{r+1}^u$ and $e'_k\geq d_u$. This implies $A\otimes
 y''\geq e'$ hence $A\otimes y''=e'$.

Case 3. We will show that if
$a_{i_j^ui_{j+1}^u}=\min\limits_{(t,l)\in c_u}
a_{tl}=x^\oplus_{i_{j+1}^u}(A)=f_{i_{j+1}^u}$ and
$\ovex_{i_{j+1}^u}> x^\oplus_{i_{j+1}^u}(A)$ then the system
$A\otimes y=f_\ovex$ has at least two solutions: $y'=f_\ovex$ and
$y''=(y''_1,\dots,y''_n)^T$, where
$$y''_i=\begin{cases}
\ovex_i & \text{  if  }i=i_{j+1}^u\\
f_i & \text{  otherwise.  }
\end{cases}$$

Observe that the vectors $y',y''$ are different in the $i_{j+1}^u$th
position:
$$y''_{i_{j+1}^u}=\ovex_{i_{j+1}^u}>x^\oplus_{i_{j+1}^u}(A)=f_{i_{j+1}^u}=y'_{i_{j+1}^u}.$$
Since $A$ is level $\gamma(A,\ovex)$-permutation and $f_{\ovex}\geq
\gamma^*(A,\ovex)$ we have  $(A\otimes y'')_{i_t^u}= a_{i_t^u
i_{t+1}^u}\otimes f_{i_{t+1}^u}=f_{i_t^u}$ for each $t\neq j$. As
for the case of $j$, recalling that
$a_{i_j^ui_{j+1}^u}=f_{i_{j+1}^u}$ and
$f_{i_{j+1}^u}<\ovex_{i_{j+1}^u}$ we obtain the following equalities
\begin{equation*}
\begin{split}
(A\otimes y'')_{i_{j}^u}=&\bigoplus\limits_{k\neq i_{j+1}^u}a_{i_j^uk}\otimes y''_k\oplus a_{i_j^ui_{j+1}^u}\otimes y''_{i_{j+1}^u}
=\bigoplus\limits_{k\neq i_{j+1}^u}a_{i_j^uk}\otimes f_k\oplus a_{i_j^ui_{j+1}^u}\otimes \ovex_{i_{j+1}^u}
=\\
&\bigoplus\limits_{k\neq i_{j+1}^u}a_{i_j^uk}
\otimes f_k\oplus (a_{i_j^ui_{j+1}^u}\otimes f_{i_{j+1}^u})\otimes
\ovex_{i_{j+1}^u}=\bigoplus\limits_{k\in N}a_{i_j^uk}\otimes f_k=f_{i_j^u}.
\end{split}
\end{equation*}
Here we have used that the equality $a_{i_j^u i_{j+1}^u} =f_{i_{j+1}^u}$
and the inequality $f_{i_{j+1}^u}<\ovex_{i_{j+1}^u}$, both following from the conditions describing Case 3.

\medskip
The ``if'' part: Suppose that $A$ is an $\mbf{X}$-conforming matrix
and we shall show that $(\forall x\in
V(A)\cap\mbf{X})[|S(A,x)\cap\mbf{X}|=1].$ For the contrary suppose
that $(\exists x\in V(A)\cap\mbf{X})[|S(A,x)\cap\mbf{X}|>1]$. By
\lemref{l:eigspace} (iii) $x=(x_1,\dots,x_n)^T,\ x_i=\alpha_u\in
[e_i, f_i]\text{ for }i\in c_u,\,1\leq u\leq k$ and there is a
solution $y'\neq x$ of the system $A\otimes y=x $. Then there is
$j\in N$ such that $x_j\neq y'_j$. We shall consider three
possibilities: (i) $ y'_j<e_j$, (ii) $f_j< y'_j$,  (iii) $y'_j\in
[e_j, f_j]$.

(i) $ y_j'<e_j$. Since $A$ is level $\gamma(A,\ovex)$-permutation
there is $p\in N$ such that $a_{pj}\geq \gamma(A,\ovex)$, so that we
can substitute $p$ for $i_j$ and $j$ for $i_{j+1}$ in
Definition~\ref{def:Xconf} of $\mbf{X}$-conforming matrix. As
$\undx_j\leq y'_j<e_j$ by condition (i) of that definition we have
that $\undx_j<e_j\Rightarrow a_{pt}<e_p=e_j$ for $t\neq j$ and we
obtain
$$x_p=(A\otimes y')_p=\bigoplus\limits_{t\neq j}a_{pt}\otimes y'_t\oplus a_{pj}\otimes y_j'<e_p\leq x_p,$$
which is a contradiction.

(ii) $ f_j<y_j'$.  As $A$ is level $\gamma(A,\ovex)$-permutation
there is $p\in N$ such that $a_{pj}\geq \gamma(A,\ovex)$, and by
Remark~\ref{r:f} we have $a_{pj}=\bigoplus\limits_{k\in N}a_{pk}\geq
f_p(=f_j)$. We consider two possibilities:

1. $a_{pj}>f_p=f_j.$ Then we obtain the following
$$x_p=(A\otimes y')_p=\bigoplus\limits_{k\neq j}a_{pk}\otimes y'_k\oplus a_{pj}\otimes y_j'\geq a_{pj}\otimes y_j'>f_p\geq x_p,$$
and this is a contradiction.

2. $a_{pj}=f_p=f_j.$ At first we shall prove the following claim.\\

Claim. If $A$ is level $\gamma(A,\ovex)$-permutation and
$x^{\oplus}_r(A)=a_{rs},\ (r,s)\in c_u$ then
$a_{rs}=\min\limits_{(k,l)\in c_u}a_{kl}.$

{\it Proof of Claim.} For the contrary, suppose that $A$ is level
$c(A)$-permutation, $a_{rs}=x^{\oplus}_r(A)$ and
$a_{rs}>\min\limits_{(k,l)\in c_u}a_{kl}=a_{\alpha\beta}.$ Using
$$a_{\alpha t}<c(A)\leq a_{\alpha\beta}<a_{rs}=x^{\oplus}_r(A)=
x^{\oplus}_\beta(A),\quad t\neq \beta,$$ we obtain
$$x^{\oplus}_r(A)=x^{\oplus}_\alpha(A)=(A\otimes x^{\oplus}(A))_\alpha=
a_{\alpha\beta}\otimes
x^{\oplus}_\beta(A)=a_{\alpha\beta}<a_{rs}=x^{\oplus}_r(A).$$
This is a contradiction. Note that the equalities
$x^{\oplus}_\alpha(A)=x^{\oplus}_\beta(A)=x^{\oplus}_r(A)$ follow
from
Lemma~\ref{l:eigspace} (since $x^{\oplus}(A)$ is an eigenvector). \\

Now we will continue to analyze ``(ii), Case 2''. The assumptions
$a_{pj}=f_p=f_j$ and $ f_j<y_j'$ imply the inequalities
$a_{pj}=f_p=f_j\leq \min(\ovex_j,\, x^{\oplus}_j(A))<y'_j\leq
\ovex_j.$ Together with $a_{pj}\geq
x_p^{\oplus}(A)=x_j^{\oplus}(A)$, following from the fact that
$x^{\oplus}(A)$ is an eigenvector.
 Thus we have the equality
$a_{pj}=x^{\oplus}_j(A)$ and by Claim we obtain $a_{pj}=\min\limits_{(k,l)\in c_u}a_{kl}.$ Then  by the
definition of $\mbf{X}$-conforming matrix we get
$$a_{pj}=f_j=x_j^{\oplus}(A)\Rightarrow \ovex_j\leq x_j^{\oplus}.$$
We conclude the proof by the following  contradiction
$$\ovex_j\leq x_j^{\oplus}=f_j<y'_j\leq \ovex_j.$$

(iii)  $y'_j\in [e_j, f_j]$. As we also assumed $x_j\neq y'_j$, we
shall analyze two possibilities: $x_j<y'_j$ and $x_j>y'_j$.

Let $x_j>y'_j$. Using Remark~\ref{r:f} and Lemma~\ref{l:eigspace},
we obtain $a_{pj}=\bigoplus\limits_{k\in N}a_{pk}\geq f_p\geq
x_p=x_j$. By  the definition of $\mbf{X}$-conforming matrix we have
that $a_{pk}\leq e_p(=e_j\leq y'_j<x_j=x_p)$ for $k\neq j$ These
inequalities imply
$$x_p=(A\otimes y')_p=\bigoplus\limits_{k\neq j}a_{pk}\otimes y'_k\oplus a_{pj}\otimes y_j'=a_{pj}\otimes y_j' =y'_j< x_p,$$
which is a contradiction.\\

Let $x_j<y'_j$. Using Remark~\ref{r:f}, \lemref{l:eigspace} (i) and
the conditions $y'_j\in[e_j,f_j]$ and $x_j<y_j'$, we obtain that
$$a_{pj}=\bigoplus\limits_{k\neq j}a_{pk}\geq f_p(=f_j\geq
y'_j>x_j=x_p).$$ This implies that
$$x_p=(A\otimes y')_p=\bigoplus\limits_{k\neq j}a_{pk}\otimes y'_k\oplus a_{pj}\otimes y_j'\geq a_{pj}\otimes y_j'>x_p\otimes x_p=x_p,$$
which is a contradiction.
\end{proof}

\begin{remark}
\label{r:complexity} {\rm \thmref{Th_WXRmaxminplus2} implies that in
the case when $\undx<c^*(A)$ and $\max_{i\in N} \undx_i<\min_{i\in
N}\ovex_i$, the complexity of checking that a given matrix $A$ has
$\mbf{X}$-simple image eigenspace for a given interval vector
$\mbf{X}$ requires $O(n^2 \log n)$ arithmetic operations.}
\end{remark}

\section{Upwardness of $\mbf{X}$-simple image eigenvectors}
\label{s:upward}

In this section we will prove that $\mbf{X}$-simple image
eigenvectors have the following property: if $\alpha\otimes x$ is an
$\mbf{X}$-simple image eigenvector, then so is $\beta\otimes x$ for
every $\beta\geq\alpha$.

We shall first generalize some basic results concerning a system of
max-min linear equations $A\otimes x=b$  (see \cite{c2},\cite{z})
when the solution set is restricted to an interval $\mbf{X}$. We
follow here the basic theory of systems $A\otimes x=b$ over max-min
algebra developed in~\cite{z}. For a different exposition of the
same theory see, e.g.,~\cite{c2} (in particular, $M_j(A,b)$, as
defined below, corresponds to $I_j(A,b)\cup K_j(A,b)$ in~\cite{c2}).

For any $j\in N$ denote
 $$x^*_j(A,b)=\min\{b_i;\ a_{ij}>b_i \},$$
   whereby  $\min\emptyset=I$ by definition. Further  denote
$$M_j(A,b)=\{i\in N;a_{ij}\otimes x^*_j(A,b)=b_i\},$$
$$S(A,b)=\{x\in \BB(n);\ A\otimes x=b\}.$$

Unique solvability can be characterized using the notion of minimal
covering. If $D$ is a set and ${\cal E}\subseteq {\cal P}(D)$ is a
set of subsets of $D$, then ${\cal E}$ is said to be a covering of
$D$, if $\bigcup{\cal E}=D$ and a covering ${\cal E}$ of $D$ is
called minimal, if $\bigcup({\cal E}-{F})\neq D$ holds for every
$F\in {\cal E}$.

\begin{theorem}{\rm\cite{c2},\cite{z}}\label{t:c1}
Let   $A\in \BB(n,n)$ be a matrix  and $b\in \BB(n)$ be a vector.
Then the following conditions are equivalent:
\begin{enumerate}
\item
$S(A,b)\neq \emptyset$,
\item
$x^*(A,b)\in S(A,b)$,
\item
$\bigcup\limits_{j\in N}M_j(A,b)=N.$
\end{enumerate}

\end{theorem}

\begin{theorem}{\rm\cite{c2},\cite{z}}\label{t:c3}
Let   $A\in \BB(n,n)$ be a matrix  and $b\in \BB(n)$ be a vector.
Then $S(A,b)=\{x^*(A,b)\}$ if and only if
\begin{enumerate}
\item
$\bigcup\limits_{j\in N}M_j(A,b)=N,$
\item
$\bigcup\limits_{j\in N'}M_j(A,b)\neq N \text{ for any
}N'\subseteq N, N'\neq N.$
\end{enumerate}

\end{theorem}

Now we shall formulate a generalized (interval) version of above
results. Let  $ \mbf{X}$ be an interval vector, $A\in \BB(n,n)$ and
$x,b\in \mbf{X}$. Without loss of generality, we can suppose that
$b_i>\undx_i$ for all $i\in N$, for the following reason. If
$b\geq\undx$ denote
by $N_\undx=\{i\in N;b_i=\undx_i\}.$ Then any solution $x$ of $A\otimes x=b$ has $x_j=\undx_j$ for all $j\in D_i=\{k\in N;\;a_{ik}>\undx_i\}$. Thus we can delete the equations with indices from $N_\undx$ and columns of $A$ with indices from $\bigcup\limits_{i\in N_\undx}D_i$   and the solutions of the original and reduced systems correspond to each other by putting $x_j=\undx_j$ for  each $i\in N_\undx$. Notice that if the system $A\otimes x=b$ is solvable and $\undx_k\neq \undx_l$ then $ D_k\cap D_l=\emptyset.$ \\

Now we shall redefine the vector $x^*(A,b)$ and then we can
reformulate the assertions of above theorems for $x\in\mbf{X}$ and
$b\in\mbf{X}$. Notice that if we consider $\undx$ instead of the
vector $(O,\dots O)^T$ and $\ovex$ instead of the vector
$(I,\dots,I)^T$ the proofs of the next three theorems are similar to
the proofs of above theorems.

Let $ \mbf{X}$ be an interval vector   and $A\otimes x=b>\undx$ be a
system of $(\max,\min)$ linear equations. For any $j\in N$ denote
 $$\tilde x^*_j(A,b)=\min\{b_i;\ a_{ij}>b_i \},
 \text{ whereby  }\min\emptyset=\ovex_j.$$ Further denote
$$\tilde M_j(A,b)=\{i\in N;a_{ij}\otimes \tilde x^*_j(A,b)=b_i\},$$
$$\tilde S(A,b)=\{x\in \mbf{X};\ A\otimes x=b\}.$$

\begin{theorem}\label{t:z}
Let   $A\in \BB(n,n)$ be a matrix  and $b\in \mbf{X}$ be a vector.
Then the following conditions are equivalent:
\begin{enumerate}
\item
$\tilde S(A,b)\neq \emptyset$,
\item
$\tilde x^*(A,b)\in \tilde S(A,b)$,
\item
$\bigcup\limits_{j\in N}\tilde M_j(A,b)=N.$
\end{enumerate}
\end{theorem}

\begin{theorem}\label{t:c31}
Let   $A\in \BB(n,n)$ be a matrix   and $b\in  \mbf{X}$ be a vector.
Then $\tilde S(A,b)=\{\tilde x^*(A,b)\}$ if and only if
\begin{enumerate}
\item
$\bigcup\limits_{j\in N}\tilde M_j(A,b)=N,$
\item
$\bigcup\limits_{j\in N'}\tilde M_j(A,b)\neq N \text{ for any
}N'\subseteq N, N'\neq N.$
\end{enumerate}

\end{theorem}

We will now state and prove the main result of this section.

\begin{theorem}
Let $x\in V(A)$, $\alpha\in [\bigoplus\limits_{i\in N}\undx_i,\bigotimes\limits_{i\in N}\ovex_i]$ and $ \alpha\otimes x$ is $\mbf{X}$-simple image eigenvector. Then $\beta\otimes x$ is $\mbf{X}$-simple image eigenvector   for   $\beta\geq\alpha.$
\end{theorem}
\begin{proof} Suppose that $x\in V(A)$, $\alpha\leq\beta$ and $|\tilde S(A,\alpha\otimes x)|=1.$ To show the assertion it suffices  to prove that
$\bigcup\limits_{j\in N}\tilde M_j(A,\beta\otimes x)\subseteq\bigcup\limits_{j\in N}\tilde M_j(A,\alpha \otimes x)$. The reason is that if $\bigcup\limits_{j\in N}\tilde M_j(A,\alpha \otimes x)$ is a minimal covering and $\bigcup\limits_{j\in N}\tilde M_j(A,\beta\otimes x)\subseteq\bigcup\limits_{j\in N}\tilde M_j(A,\alpha \otimes x)$ then $\bigcup\limits_{j\in N}\tilde M_j(A,\beta\otimes x)$ is a minimal covering as well. Notice that $\bigcup\limits_{j\in N}\tilde M_j(A,\beta\otimes x)$
is a covering because of $\beta\otimes x\in S(A,\beta\otimes x)$.

Claim: If $\bigoplus\limits_{i\in N}\undx_i\leq\alpha,\beta\leq\bigotimes\limits_{i\in N}\ovex_i$ then
 $\tilde x^*(A,\alpha\otimes x)\leq \tilde x^*(A,\beta\otimes x)$ and

$
\tilde x^*_j(A,\alpha\otimes x)=
\begin{cases}
\tilde x^*_j(A,\beta\otimes x),\text{  if  }\tilde x^*_j(A,\alpha\otimes x)=\ovex_j\\
\alpha\otimes \tilde x^*_j(A,\beta\otimes x), \text{  otherwise.  }
\end{cases}
$

Proof of claim.  Let $j\in N$ be a fixed index and by the definition of $\tilde x^*(A,\alpha\otimes x)$ we get
$$ x^*_j(A,\alpha\otimes x)=\min\{\alpha\otimes x_i;\ a_{ij}>\alpha\otimes  x_i \}\leq$$
 $$
\min\{\beta\otimes x_i;\ a_{ij}>\beta\otimes x_i \}= x^*_j(A,\beta\otimes x)
$$
 because of $a_{ij}>\beta\otimes x_i\geq  \alpha\otimes x_i.$\\

The equality
$\tilde x^*_j(A,\alpha\otimes x)=\ovex_j$ together with the inequality
$\tilde x^*(A,\alpha\otimes x)\leq \tilde x^*(A,\beta\otimes x)$ imply
$ \tilde x^*_j(A,\beta\otimes x)=\ovex_j=\tilde x^*_j(A,\alpha\otimes x).$\\

Suppose that $\tilde x^*_j(A,\alpha\otimes x)<\ovex_j.$ Then by the definition of $\tilde x^*(A,\alpha\otimes x)$ we get
$$\tilde x^*_j(A,\alpha\otimes x)=\min\{\alpha\otimes x_i;\ a_{ij}>\alpha\otimes x_i(=\alpha\otimes x_s) \}.$$
Consider two possibilities:

1. $\tilde x^*_j(A,\beta\otimes x)=\ovex_j.$ Then we obtain
$$\tilde x^*_j(A,\alpha\otimes x)=\alpha\otimes x_s<a_{sj}\leq \beta\otimes x_s.$$
Thus the inequality $\alpha\otimes x_s< \beta\otimes x_s$ implies
$\alpha<\beta$ and $\alpha<x_s$ and we get
$$\tilde x^*_j(A,\alpha\otimes x)=\alpha\otimes x_s=\alpha=\alpha\otimes\ovex_j=\alpha\otimes \tilde x^*_j(A,\beta\otimes x).$$

2. $\tilde x^*_j(A,\beta\otimes x)<\ovex_j.$ There is $r\in N$ such
that
$$\alpha\otimes x_s=\tilde x^*_j(A,\alpha\otimes x)\leq \tilde x^*_j(A,\beta\otimes x)=\beta\otimes x_r.$$ Notice that if
 $\alpha\otimes x_s<\beta\otimes x_r$ then $\alpha\leq x_s$
 (if $\alpha> x_s$ then $x_s=\alpha\otimes x_s=\beta\otimes x_s<\beta\otimes x_r$ and
 this is a contradiction with $x^*_j(A,\beta\otimes x)=\beta\otimes x_r$).
 Hence
$$\tilde x^*_j(A,\alpha\otimes x)=\alpha\otimes x_s=\alpha\otimes(\alpha\otimes x_s)=\alpha\otimes(\beta\otimes x_r)=\alpha\otimes \tilde x^*_j(A,\beta\otimes x).$$
Now we shall prove the inclusion
$\bigcup\limits_{j\in N}\tilde M_j(A,\beta\otimes x)\subseteq\bigcup\limits_{j\in N}\tilde M_j(A,\alpha \otimes x)$. \\

Let $k\in \tilde M_j(A,\beta\otimes x)$, i.e., $a_{kj}\otimes \tilde x^*_j(A,\beta\otimes x)=\beta\otimes x_k(\geq \alpha\otimes x_k).$
We shall consider two cases.\\

Case 1. $\tilde x^*_j(A,\alpha\otimes x)=\ovex_j$. In this case we
have that $a_{lj}\leq\alpha\otimes x_{\ell}$ for all $\ell$, and in
particular,
$$a_{kj}\leq \alpha\otimes x_k\Rightarrow
a_{kj}\otimes \tilde x^*_j(A,\alpha\otimes x)\leq\alpha\otimes
x_k.$$ For the opposite inequality observe that
$$a_{kj}\otimes \tilde x^*_j(A,\alpha\otimes x)=a_{kj}\otimes \tilde x^*_j(A,\beta\otimes x)=\beta\otimes x_k\geq \alpha\otimes x_k.$$

Case 2. $\tilde x^*_j(A,\alpha\otimes x)<\ovex_j$. This case follows from the fact that
$$a_{kj}\otimes \tilde x^*_j(A,\alpha\otimes x)=a_{kj}\otimes (\alpha\otimes \tilde x^*_j(A,\beta\otimes x))=\alpha\otimes(\beta\otimes x_k)= \alpha\otimes x_k.$$\end{proof}

\end{document}